\documentclass{article}[12pt]
\usepackage{theorem}
\usepackage{eurosym}
\usepackage{amssymb}
\usepackage{float}
\usepackage{amsmath}
\usepackage{amsfonts}
\usepackage[american]{babel}
\usepackage{verbatim}
\usepackage{geometry}

\newcommand{\R}        {\mathbb {R}}

\newcommand{\eps}      {\epsilon}
\newcommand{\lap}      {\bigtriangleup}

\newcommand{\grad}     {\nabla}
\newcommand{\noi}      {\noindent}
\newcommand{\del}      {\partial}

\newtheorem{theorem}{Theorem}

{\theorembodyfont{\rmfamily}

}

\newtheorem{lemma}[theorem]{Lemma}

{\theorembodyfont{\rmfamily}

}

\newenvironment{proof}[1][Proof]{\noindent\textbf{#1} }{\ \rule{0.5em}{0.5em}}

\begin{document}

\title{Extremal function for Moser-Trudinger type Inequality with Logarithmic weight}
\maketitle
\date{}

\centerline{\scshape   Prosenjit Roy }
\medskip
{\footnotesize

  \centerline{ Tata Institute of Fundamental Research, Center for Applicable Mathematics, Bangalore, 560065, India.}

}

\smallskip

\begin{abstract} 
On the space of weighted radial Sobolev space, the following  generalization of Moser-Trudinger type inequality was established by Calanchi and Ruf  in dimension 2 : If $\beta \in [0,1)$ and $w_0(x) = |\log |x||^\beta $ then
$$ \sup_{\int_B |\grad u|^2w_0 \leq 1 ,  u \in H_{0,rad}^1(w_0,B)} \int_B e^{\alpha u^{\frac{2}{1-\beta}}} dx  < \infty,$$ if and only if 
$\alpha \leq  \alpha_\beta = 2\left[2\pi (1-\beta) \right]^{\frac{1}{1-\beta}}.$  We prove the existence of an extremal function for the above inequality  for the critical case when 
$\alpha = \alpha_\beta$   thereby generalizing the result of Carleson-Chang who proved the case when $\beta =0$.
\end{abstract}

\section{Introduction}\label{section:introduction}
Let $\Omega$ be a smooth bounded domain in $\R^n$. For $p < n$ it is well known from Sobolev embedding that the space 
$$W_0^{1,p}(\Omega) \hookrightarrow L^p(\Omega) $$ continuously, if $1\leq p \leq  p^* = \frac{np}{n-p}.$  The Moser-Trudinger inequality concerns about the borderline case, that is when $p = n$.  In this case 
$$W_0^{1,n}(\Omega) \hookrightarrow L^p(\Omega)$$ for all $ p \in [1, \infty)$. The embedding is not true for case  $p =\infty$. For $n=2$ one can take the function 
$$f(x) = \log \left( 1 -\log |x| \right)$$ in unit ball $B$ centered at origin, then it is easy to check that $f \in W_0^{1,2}(B)$ but clearly it is not in $L^\infty(B)$. Hence one may look for maximal growth function $g$ such that
$$\int_\Omega g(u) < \infty, \  \forall u \in W_0^{1,n}(\Omega).$$ 
It was shown by Trudinger in \cite{tru}, that such a $g$ should be of the form 
$$g(t) = e^{t^{\frac{n}{n-1}}}.$$  
There was a further improvement by Moser who  proved the following theorem:
\begin{theorem}\emph{[Moser]}
\label{moser}
Let $u \in W_0^{1,n}(\Omega)$ with 
$$\int_\Omega |\grad u |^n \leq 1.$$
Then there exist a constant $C > 0$ depending only on $n$, such that 
\begin{equation}
\label{mo}
\int_\Omega e^{\alpha u^{\frac{n}{n-1}}} \leq C |\Omega |, \hspace{3mm}   \  \textrm{if}  \  \alpha \leq \alpha_n := n \omega_{n-1}^{\frac{1}{n-1}}
\end{equation}
where $w_{n-1}$ is the $(n-1)$ dimensional surface measure of the unit sphere and the above result is false if $\alpha > \alpha_n$. Here $|\Omega | $ denotes the $n$ dimensional  Lebesgue measure of the set $\Omega$. 
\end{theorem}
Then the next natural question was  if there exist an extremal function  for the inequality in \eqref{mo}. Carleson-Chang \cite{Carleson-Chang} showed that the answer
 to the above question is positive if the domain is a ball.  In \cite {Struwe 1} Struwe proved the case when the domain is close to a ball. Flucher in \cite{Flucher}  provided a positive answer for the case of  any general domain $\Omega$ in $2$ dimensions. 
 The higher dimension case for general domains was done by Lin in \cite{LIN}. 
 Moser-Trudinger type inequality had been an interesting  topic of research for several authors. We list a few of them  \cite{Adi-Sandeep, Adi-Tintarev, ca,  Csato-Roy, Csato-Roy2, lu2, lu,
  Lions, Malchiodi-Martinazzi, manchini, Struwe 1}  and the references there in for the available literature in this direction.\smallskip

Let $n=2$ and  $$w_0 = |\log|x||^\beta$$ be defined on the unit ball $B$ and $H_0^1(w_0, B)$ denotes the usual weighted 
Sobolev space defined as the completion of $C_c^\infty(B)$(the space of smooth functions with compact support) functions with respect to the norm 
$$||u||_{w_0} := \left(\int_{B} |\grad u|^2w_0 dx \right)^{\frac{1}{2}}.$$
The subspace of radial functions in $H_0^1(w_0,B)$ is denoted by $H_{0, rad}^1(w_0,B)$.\smallskip

The following theorem by Calanchi-Ruf  in \cite{ruc} generalizes the Moser-Trudinger inequality for balls. The work of this paper is based on this generalization.
\begin{theorem}
\label{ruf}
Let $\beta \in [0,1)$ and $n=2$, then 
\begin{equation}
\label{ruff}
\sup_{||u||_{w_0} \leq 1,  u \in H_{0, rad}^1(w_0,B)} \int_{B} e^{\alpha u^{\frac{2}{1-\beta}}} dx < \infty
\end{equation}
if and only if $\alpha \leq \alpha_\beta := 2\left[2\pi(1-\beta)^{\frac{2}{1-\beta}}\right] $ and the inequality above is false if $\alpha > \alpha_\beta$.
\end{theorem}\smallskip

They have also obtained the optimal Moser-Trudinger type inequality for the case $\beta = 1$ which we do not mention here. 
For $\beta = 0$ the above theorem is precisely Moser-Trudinger inequality for balls in 2 dimensions. This is obvious after using a symmetrization argument.
 In this work we are concerned with the existence of extremal function for the inequality in \eqref{ruff} for the  critical case  i.e. $\alpha = \alpha_\beta$.
In the sub critical case $(\alpha < \alpha_\beta )$ the issue of existence of an extremizer 
is not very difficult, as one can use Vitalli's convergence theorem to pass through the limit. This issue is  adressed in \cite{ruc}.  Also very recently in \cite{cr} Calanchi-Ruf had also established such  logarithmic Moser-Trudinger type inequalities in higher dimensions.
\smallskip

For each $\beta \in [0,1)$, let $J_\beta : H_{0,rad}^1( w_0, B) \rightarrow \R$ defined by   $$J_\beta(u) : =  \frac{1}{|B|} \int_{B} e^{\alpha_\beta u^{\frac{2}{1-\beta}}} dx $$
denotes Logarithmic Moser-Trudinger functional. The following is our main result:

\begin{theorem}\emph{[Main Result]}
\label{pro}
There exist $u_\beta \in H_{0, rad}^1(w_0,B)$  such that
$$ J_\beta(u_\beta)  = \sup_{||u||_{w_0} \leq 1,  u \in  H_{0,rad}^1( w_0, B) } J_\beta(u),$$
for all  $\beta\in [0,1)$.
\end{theorem}

 The main difficulty  in proving the existence of extremizer lies in the fact that the functional $J_\beta$ is not continuous with respect to the weak 
 convergence of the space $H_{0,rad}^1( w_0, B)$.  The following sequence $w_k \in H_{0,rad}^1( w_0, B)$ (power of Moser sequence), defined as   
 \begin{equation}\label{Main Probldem}
  w_k(x) = \alpha_\beta^{\frac{\beta -1}{2}}\left\{
 \begin{array}{ll}
 k^{\frac{1-\beta}{2}} \hspace{8.7mm}&\textrm{in} \   0 \leq |x| \leq e^{-\frac{k}{2}}, \\
   \left( \frac{-2\log|x|}{\sqrt{k}} \right)^{1-\beta} \hspace{16.9mm} &\textrm{on}   \   e^{-\frac{k}{2}} \leq |x| \leq 1 , 
\end{array}
\right.
\end{equation}
 has the property that  $w_k \rightharpoonup 0$ in $H_{0,rad}^1( w_0, B)$ but   $J_\beta(w_k) \nrightarrow J_\beta(0).$ We refer to Section \ref{s1} for the proof of the last statement. \bigskip

\underline{\textit{ Definition }(Concentration): } \   A sequence of functions $u_k  \in H_{0,rad}^1( w_0, B)$  is said to concentrate at $x=0$, denoted  by 
  $$|\grad u_k|^2w_0 \rightharpoonup \delta_0,$$ if $||u_k||_{w_0} \leq 1$ and 
for any given $1 > \delta > 0$,  
$$\int_{B \setminus B_\delta} |\grad u_k|^2w_0 \rightarrow 0,$$ \bigskip
where $B_\delta$ denotes the ball of radius $\delta$ at origin.

\underline{\textit{ Definition }(Concentration Level $J_\beta^\delta(0)$): } \smallskip

$$J_\beta^\delta(0) := \sup_{\{w_m\}_m \in  H_{0,rad}^1( w_0, B)}\left\{ \limsup_{m \rightarrow 
\infty} J(w_m) \ \big|  \  |\grad w_m|^2w_0 \rightharpoonup \delta_0 \right\}.$$
\smallskip

The method of the proof of our main result follows similar idea as it is done in \cite{Carleson-Chang}. First in Lemma \ref{cx} we show that a maximizing sequence 
can loose compactness only if it  concentrates at $0$. Then in Lemma \ref{v} the concentration level is explicitly calculated. It is shown that for all $\beta \in [0,1)$
$$(1)  \hspace{4mm} J_\beta^\delta(0) \leq 1+ e.$$
Note that the right hand side of the above inequality is independent of $\beta$.\smallskip

Finally we finish the proof by showing that for all $\beta \in  [0,1)$ one can find $v_\beta \in H_{0,rad}^1( w_0, B)$ such that 
$$(2)   \hspace{4mm} J_\beta(v_\beta) > 1+ e.$$\smallskip

As a direct application of the above theorem one obtains existence of radial solution of the following nonlinear elliptic Dirichlet (mean field type) problem for $\beta \in [0.1)$:

\begin{equation*}
  \left\{
 \begin{array}{ll}
 -\rm{div}\left( w_0 \grad u \right) =  \frac{u^{\frac{\beta+1}{\beta-1}} e^{\alpha_\beta u^{\frac{2}{1-\beta}}}}{\int_B u^{\frac{\beta+1}{\beta-1}} e^{\alpha_\beta u^{\frac{2}{1-\beta}}}}\hspace{8.7mm}&\textrm{in} \ B , \\
  u = 0   \hspace{16.9mm} &\textrm{on}  \  \del B,\\
  u > 0 \hspace{16.9mm} &\textrm{on}  \   B.
\end{array}
\right.
\end{equation*}
\section{Some  supporting results for the proof of main theorem }\label{s1}
At first we will deduce an equivalent formulation of the problem with which we will work in this paper.
 Let $\gamma = \frac{1}{1-\beta}$. For  $u \in  H_{0,rad}^1( w_0, B)$  first change the variable as
 \begin{equation}\label{sss} |x| = e^{\frac{-t}{2}}\end{equation}  and set
\begin{equation}\label{ssss} \psi(t) =\alpha_\beta^{\frac{1}{2\gamma}} u(x).\end{equation}
Then the functional changes as 
\begin{equation}\label{pl}
I_\beta(\psi) := \int_0^\infty e^{\psi^{2\gamma}(t) - t} dt = \frac{1}{|B|}\int_B e^{\alpha_\beta u^{2\gamma}}dx = J_\beta(u)
\end{equation}
and the weighted gradient norm  changes as 
$$\Gamma(\psi) :=  \int_0^\infty \frac{|\psi'|^2t^\beta}{1-\beta} dt = \int_B |\grad u|^2|\log|x||^\beta dx.$$
For $\delta \in (0,1]$, define 
$$ \tilde{\Lambda_\delta} := \left \{  \phi \in C^1(0, \infty)  \   \big| \    \phi(0) = 0 ,   \ \Gamma(\phi)\leq \delta \right\}$$ and now since
$$\sup_{||u||_{w_0} \leq 1,  u \in H_{0,rad}^{1}(w_0, B)} \int_{B} e^{\alpha_\beta u^{\frac{2}{1-\beta}}} dx = \sup_{||u||_{w_0} \leq 1, u \in H_{0,rad}^{1}(w_0, B), smooth}
 \int_{B} e^{\alpha_\beta u^{\frac{2}{1-\beta}} }dx,$$
the problem reduces in finding $\psi_0 \in \tilde{\Lambda_1}$ such that
\begin{equation}\label{gf}
M_\beta := I_\beta(\psi_0) = \sup_{\psi \in \tilde{\Lambda_1}} I_\beta(\psi).
\end{equation}  

\begin{lemma}
\label{sd}Let $w_k$ be as in \eqref{Main Probldem}, then for all $\beta \in [0,1)$,
$$ \liminf_{k \rightarrow \infty} J_\beta (w_k) > 1 + \frac{1}{e} > J_\beta(0)  = 1.$$
\end{lemma}
\begin{proof}
It is clear that $w_k \rightharpoonup 0$ in $ H_{0,rad}^1( w_0, B)$.  Let $\psi_k(t) = \alpha_\beta^{\frac{1}{2\gamma}} w_k(x)$ and $|x| = e^{-\frac{t}{2}}$. 
Then in view of \eqref{pl}, it is enough to show that 
$$\liminf_{k \rightarrow \infty}  I_\beta(\psi_k)> 1.$$
Then 
$$I_\beta(\psi_k)  = \int_0^ke^{\frac{t^2}{k} -t}dt + \int_k^\infty e^ke^{-t}dt = k \int_0^1e^{k(t^2-t)}dt + 1.$$
Note that the function $e^{t^2-t}$ is monotone decreasing on $[0,\frac{1}{2}]$, monotone increasing on $[\frac{1}{2},1]$  and strictly positive. 
Therefore one has 
$$k \int_0^1 e^{k(t^2-t)} > k e^{k(\frac{1}{k^2} -\frac{1}{k}) }\frac{1}{k} =e^{\frac{1}{k}}e^{-1}.$$
This implies that 
$$\liminf_{k \rightarrow \infty} I_\beta(\psi_k) > 1 + \frac{1}{e}.$$
This completes the proof of the lemma.
\end{proof}\bigskip

Let $\tilde{g}_m$ be a maximizing sequence i.e. $J_\beta(\tilde{g}_m) \rightarrow M_\beta$. Since $$\int_B |\grad \tilde{g_m}|^2|\log|x||^\beta dx \leq 1,$$
one can find up to a subsequence (which we again denote by $\tilde{g}_m$) and for some function $\tilde{g}_0 \in   H_{0,rad}^1( w_0, B)$

\begin{eqnarray}\label{xc}
\tilde{g}_m \rightharpoonup \tilde{g}_0 \hspace{4mm}  \textrm{in} \  H_{0,rad}^1( w_0, B), \nonumber \\
\tilde{g}_m \rightarrow \tilde{ g}_0 \hspace{4mm}  \textrm{pointwise}.
\end{eqnarray}

The next lemma  is equivalent to  concentration-compactness  alternative, for Moser-Trudinger case, by P. L. Lions in \cite{Lions}. As a consequence of the next lemma it would be enough to prove that the sequence $\tilde{g_m}$ does not concentrates at $0$, in order to pass thought the limit  in the functional.

\begin{lemma}\emph{[Concentration-Compactness alternative]} \label{cx} For any sequence $\tilde{w_m}, \tilde{ w}  \in  H_{0,rad}^1( w_0, B)$ such that $\tilde{w_m} \rightharpoonup \tilde{w} $ in $ H_{0,rad}^1( w_0, B)$, then for a subsequence, either $$(1) \  I_\beta(\tilde{w_m}) \rightarrow I_\beta(\tilde{w}),$$  or
$$(2) \  \tilde{w_m}  \textrm{  concentrates at} \   x=0 . $$

\end{lemma}

\begin{proof}
Let us assume that $(1)$ does not hold. Then  it is enough to show that for each $A > 0$,  as $m \rightarrow \infty$, it implies that 
\begin{equation*}\label{conc}
\int_{B\setminus B_{e^{\frac{-A}{2}}}}|\grad \tilde{w_m}|| \log|x||^\beta dx = \int_0^A \frac{|w_m^{'}|^2t^\beta}{1-\beta} dt \rightarrow 0 
\end{equation*}
where $$\alpha_\beta^{\frac{1}{2\gamma}}\tilde{w_m}(x)= w_m(t), \ \tilde{w}(x) = w(t) \ \textrm{and} \ |x| = e^{-\frac{t}{2}}.$$
We argue by contradiction. Then there exists some $A > 0$ and $\delta > 0$ with  $$\int_0^A \frac{|w_m^{'}|^2t^\beta}{1-\beta} dt \geq \delta, \hspace{4mm} \textrm{for all} \  m \geq m_0,$$ 
for some $m_0$. Using Fundamental theorem of calculus and H\"older's inequality, we obtain for $t \geq 	A$,
\begin{multline*}\label{n}
w_m(t) - w_m(A) = \int_A^tw_m^{'}(s) = \sqrt{1-\beta}\int_A^t\frac{w_m^{'}(s) s^{\frac{\beta}{2}}}{\sqrt{1-\beta}s^{\frac{\beta}{2}}}    \\ 
\leq \sqrt{1-\beta}
 \left( \int_A^t \frac{|w_m^{'}|^2 s^\beta}{1-\beta}\right)^{\frac{1}{2}} \left( \int_A^t s^{-\beta} \right)^{\frac{1}{2}} 
 \leq \left( \int_A^t \frac{|w_m^{'}|^2 s^\beta}{1-\beta}\right)^{\frac{1}{2}} \left( t^{1-\beta} - A^{1-\beta}\right)^{\frac{1}{2}}\\
 \leq  (1-\delta)^{\frac{1}{2}}\left( t^{1-\beta} - A^{1-\beta}\right)^{\frac{1}{2}}   \leq (1-\delta)^{\frac{1}{2}} t^{\frac{1-\beta}{2}}.
\end{multline*}
Now using the inequality $w_m(A) \leq A^{\frac{1-\beta}{2}}$ for all $m$, we have for $t \geq N $, (for sufficiently large $N$ )
\begin{eqnarray}\label{1}
w_m(t)^{\frac{2}{1-\beta}} \leq  \left\{ A^{\frac{1-\beta}{2}}  + (1-\delta)^{\frac{1}{2}} t^{\frac{1-\beta}{2}} \right\}^{\frac{2}{1-\beta}}  
 \leq  A + \left(1 -\frac{\delta}{2}\right)^{\frac{1}{1-\beta}}t.
\end{eqnarray}
In the last step we have used the following inequality: If $\mu > \gamma > 0$, $p > 1$ then for sufficiently large $y \in \R$, one has 
$$(1 + \gamma y)^p \leq 1 + \mu^p y^p.$$
Therefore from \eqref{1}, we have 
$$ I_2^m(w_m) := \int_N^\infty e^{w_m(t)^{\frac{2}{1-\beta}} - t}dt \leq e^A \int_N^\infty e^{\left[\left(1 -\frac{\delta}{2}\right)^{\frac{1}{1-\beta}}-1 \right]t},$$
which can be made less than any arbitrary positive number $\eps$, after choosing $N$ large enough.  Since we know that $\tilde{w_m}$ converges pointwise to  $\tilde{w_m}$, this implies  that $w_m$ also converges pointwise to $w$.  
\smallskip

Now we split $I_\beta(w_m) = I_1(w_m) + I_2(w_m)$ where $$I_1^m(w_m) := \int_0^N e^{w_m(t)^{\frac{2}{1-\beta}} - t}dt.$$
Using the bound $w(t) \leq t^{\frac{1-\beta}{2}}$ and dominated convergence theorem, one obtains $$I_1(w_m) \rightarrow I_1(w)$$ 
and for $I_2^m(w_m)$ we already know that it can made arbitrarily small. Therefore $I_\beta(w_m) \rightarrow I_\beta(w)$ which is a contradiction.
\end{proof}\bigskip

\underline{An Inequality :} \  For any $w \in  C^1(0,\infty) $ and $ t \geq A \geq 0$,  then we get after using H\"older's inequality,  that 
\begin{equation}\label{imp}
w(t) - w(A) 
 \leq \left( \int_A^t \frac{|w'|^2 s^\beta}{1-\beta}\right)^{\frac{1}{2}} \left( t^{1-\beta} - A^{1-\beta}\right)^{\frac{1}{2}}.
\end{equation} 
We will recall this inequality several times.

\bigskip

The following lemma is proved in [\cite{Carleson-Chang} , Lemma 1].  Here we will use it without giving the proof. Let for $\delta > 0$,

$$ \Lambda_\delta := \left \{  \phi \in C^1(0, \infty)  \   \big| \    \phi(0) = 0 ,  \  \int_0^\infty |\phi {'}|^2dt \leq \delta \right\}.$$

\begin{lemma}\emph{[Carleson-Chang]}\label{cc}
For each $c > 0$, we have 
$$\sup_{\phi \in \Lambda_\delta} \int_a^\infty e^{c\phi(t) - t}dt  \ \leq \  e^{\frac{c^2\delta}{4} + 1 }.$$ 
\end{lemma}
The next lemma is a technical result that will be useful in proving Lemma \ref{v}.

\begin{lemma}\label{3}
For $ a > 0$, if $1 - \gamma \delta> 0 $ then
$$\sup_{\phi \in \tilde{\Lambda}_\delta} \int_a^\infty e^{\phi^{2\gamma}(t) - t}dt  \ \leq \frac{e^{1-a}}{(1-\gamma\delta) }e^{\phi^{2\gamma}(a) + \frac{ \gamma \phi^{2\gamma}(a) \delta}{(1-\gamma\delta )}}$$
where $\gamma = \frac{1}{1-\beta}$.
\end{lemma}
\begin{proof}
Put $w = \sqrt{1-\beta} \phi^{\gamma}$ for $\phi \in \tilde{\Lambda_\delta}$. Then easy computation gives 
$$\int_0^\infty |w'|^2dt = \int_0^\infty \frac{|\phi'|^2 \phi^{\frac{2\beta}{1-\beta}}}{1-\beta}dt \leq  \int_0^\infty \frac{|\phi'|^2 t^\beta}{1-\beta}dt \  \leq  \delta $$
and $$ I:=  \int_a^\infty e^{\phi^{2\gamma}(t) - t}dt  = \int_a^\infty e^{\gamma w^2(t) -t} dt.$$
In one of the inequality above we have used \eqref{imp} with $A=0$.  Now changing the variable as $x = t -a$ and $w(t) = \psi(x) + w(a)$ for all $x \geq 0$. Then it is easy to see that 
$$\int_0^\infty |\psi'|^2 dx\leq \delta.$$
From \eqref{imp} with $\beta =0,   \  A=0 $ we get $\psi^2(x) \leq \delta x.$ Using this the functional changes as 
\begin{eqnarray*}
I =  \int_a^\infty e^{\gamma w^2(t) -t} dt = e^{\gamma w^2(a)-a}\int_0^\infty e^{\gamma(2\psi(x)w(a) + \psi^2(x))- x}dx \nonumber\\ \leq
e^{\gamma w^2(a)-a} \int_0^\infty e^{2\gamma w(a)\psi(x) -(1-\gamma\delta)x}dx.
\end{eqnarray*}
Further changing the variable to $\chi (y) = \psi(x)$ and $y = (1-\gamma\delta)x$, we get 
$$\int_0^\infty |\chi'(y)|^2dy  = \frac{1}{1-\gamma\delta} \int_0^\infty |\psi'|^2dx \leq \frac{\delta}{1-\gamma\delta}.$$ 
Therefore $\chi \in \Lambda_{\frac{\delta}{1-\gamma\delta}}$ whenever $1-\gamma\delta > 0$. 
The functional changes as 
$$I = \frac{e^{\gamma w^2(a)-a}}{1-\gamma\delta} \int_0^\infty e^{2\gamma w(a)\chi(y) - y}dy = \frac{e^{\phi^{2\gamma}(a)-a}}{1-\gamma\delta} \int_0^\infty e^{2\gamma w(a)\chi(y) - y}dy.$$
Finally applying Lemma 4, we obtain  if $1-\gamma\delta >0 $,
$$I \leq  \frac{e^{\phi^{2\gamma}(a)-a}}{1-\gamma\delta}e^{\gamma^2 w^2(a) \frac{\delta}{1-\gamma\delta} + 1} = \frac{e}{1-\gamma\delta }e^{\phi^{2\gamma}(a)-a 
+ \frac{ \delta\gamma \phi^{2\gamma}(a) }{1-\gamma\delta}} .$$
This finishes the proof of the lemma.\end{proof}

\section{Proof on the main theorem}
Let $\tilde{f_m} \in H_{0,rad}^1( w_0, B)$ such  that $| \grad \tilde{f_m}|^2w_0 \rightharpoonup \delta_0$.    
Then we know that $\tilde{f_m} \rightharpoonup 0$ in $H_{0,rad}^1( w_0, B)$. We further consider $f_m$ such that
 $J_\beta(\tilde{f_m}) \nrightarrow J_\beta(0)=1.$  Define $f_m$, from $\tilde{f_m}$, using the same transformation introduced in \eqref{sss} and \eqref{ssss}.

\begin{lemma}\label{point}
Let $f_m$ be as  above. There exists $a_m $, the first point in  $[1, \infty)$, with  
\begin{equation}\label{po}
f_m^{\frac{2}{1-\beta}}(a_m) - a_m = - 2\log (a_m) 
\end{equation}
and also this $a_m \rightarrow \infty$ as $m \rightarrow \infty$.

\end{lemma}
\begin{proof} \underline{STEP 1:  \ Existence of  $a_m$} \smallskip

Since $f_m(t) \leq  t^{\frac{1-\beta}{2}}$, this implies that $f_m(t)^{\frac{2}{1-\beta}} - t \leq  0 $ if $t \in [0, 1)$, while $-2\log (t) > 0$ if  $t \in [0, 1)$.  Therefore 
$ f_m^{\frac{2}{1-\beta}}(t) - t < - 2\log (t)$ which implies non existence of 
such $a_m$ satisfying \eqref{po} on the interval $[0,1)$.\smallskip

Now let us assume the non existence of such $a_m$'s in the interval $[1,\infty)$. This implies that $f_m^{\frac{2}{1-\beta}}(t) - t < - 2\log (t) $ on $[1,\infty)$. Or in other words we have 
$$e^{f_m(t)^{\frac{2}{1-\beta}} - t} \leq \frac{1}{t^2}, \hspace{4mm} \textrm{if}  \   t \in [1, \infty).$$
One can use dominated  convergence theorem, with the dominating function 
\begin{equation*}\label{Main Problem}
 g(t) = \left\{
 \begin{array}{ll}
 1  \hspace{8.7mm}&\textrm{in} \ (0,1) , \\
  \frac{1}{t^2} \hspace{16.9mm} &\textrm{on}  \ [1,\infty),
\end{array}
\right.
\end{equation*}
to show that $I_\beta(f_m) \rightarrow 1$. This is a contradiction to our assumption.
\smallskip

\underline{STEP 2 : \ $a_m \rightarrow \infty$} \smallskip 

 Given $K$ arbitrary large number. It suffices to show that  for all $m \geq m_0 $, one has $a_m \geq K$. First choose $\eta $ small, such that 
\begin{equation}\label{ll}
\eta t < t - 2 \log(t), \hspace{4mm} \textrm{for all}  \   t \in [0, K).
\end{equation}
Now using the last lemma we get for  $t \in [0, K)$ and $\forall \ m \geq m_0,$
$$f_m(t)^{\frac{2}{1-\beta}} \leq \left( \int_0^K \frac{|w_m^{'}|^2 t^\beta}{1-\beta} dt\right)^{\frac{1}{1-\beta}} t  < \eta t \leq t - 2 \log(t).$$
This says that $a_m > K$ forall $m \geq m_0$.
\end{proof}
\begin{lemma}\emph{[Estimate for Concentration level]}\label{v}
For $\beta \in [0,1)$ it implies that
\begin{equation}\label{2}
J_\beta^\delta(0)  \leq  1 + e.
\end{equation}

\end{lemma}
\begin{proof}
First note that it is enough to consider concentrating sequences $\tilde{f_m}$ such that $J_\beta(\tilde{f_m}) \nrightarrow J_\beta(0) =  1,$  because in this case the required inequlity \eqref{2} is already satisfied. \smallskip

\underline{Step 1 :}  
\begin{equation*}
\lim_{m \rightarrow \infty} \int_0^{a_m} e^{f_m(t)^{\frac{2}{1-\beta}}-t}dt = 1,
\end{equation*}
where $f_m$ and $a_m$ is as in the previous lemma.\smallskip

Using Lemma \ref{cx} and \eqref{imp}  we notice that $f_m \rightarrow 0$ uniformly on compact subsets of $\R^+$. Therefore for each $ A, \ \eps > 0$, we have $f_m(t)^{\frac{2}{1-\beta}} \leq \eps$ for all $t \leq  A$ and sufficiently large $m$.
Using the property of $a_m$, that  for all $t \leq a_m$ one has $f_m(t)^{\frac{2}{1-\beta}} \leq t -2\log (t)$, we get 
\begin{eqnarray*}
\int_0^{a_m} e^{f_m(t)^{\frac{2}{1-\beta}}- t}dt &=& \int_0^{A} e^{f_m(t)^{\frac{2}{1-\beta}}- t}dt + \int_A^{a_m} e^{f_m(t)^{\frac{2}{1-\beta}}- t}dt \nonumber \\
&\leq & e^\eps \int_0^A e^{-t}dt + \int_A^{a_m}e^{2\log(t)} dt \nonumber \\  &=&  e^\eps(1 -e^{-A}) + \left( \frac{1}{A} -\frac{1}{a_m}\right) \leq 1,
\end{eqnarray*} 
as 
$\eps \rightarrow 0$ and for large $A$.  For the  other way round 
$$\int_0^{a_m} e^{f_m(t)^{\frac{2}{1-\beta}}- t}dt \geq \int_0^{a_m} e^{-t}dt = 1 - e^{-a_m}  \rightarrow 1.$$

\underline{ Step 2 :}  \  We claim that
 $$\lim_{m \rightarrow \infty}\int_{a_m}^\infty  e^{f_m(t)^{\frac{2}{1-\beta}}-t}dt \leq e.$$
Set 
$$\delta_m = \int_{a_m}^\infty   \frac{|f_m'|^2 t^\beta}{1-\beta}dt. $$ Then using the relation (which is obtained from \eqref{imp}) $$f_m^{2\gamma}(t) \leq \left( \int_0^t \frac{|f_m'|^2t^\beta}{1-\beta} dt\right)^\gamma t$$
one obtains
\begin{equation}\label{wq}
\delta_m := 1 -  \int_{0}^{a_m}   \frac{|f_m'|^2 t^\beta}{1-\beta}dt  \leq 1 - \left( \frac{f_m^{2\gamma}(a_m)}{a_m} \right)^{\frac{1}{\gamma}} 
= 1 - \left( 1 -\frac{2\log(a_m)}{a_m}\right)^{\frac{1}{\gamma}}.
\end{equation}
In the last inequality we have used the property of the points $a_m$. 
\bigskip

Set $\delta = \delta_m$, $a= a_m$ and $\phi = f_m$ in Lemma \ref{3}, then clearly 
$$\int_{a_m}^\infty  e^{f_m(t)^{\frac{2}{1-\beta}}-t}dt \leq \frac{e^{K_m +1}}{1-\gamma \delta_m}$$
where 
 $$K_m = \left( f_m^{2\gamma}(a_m)-a_m \right)  
+ \frac{\delta_m\gamma f_m^{2\gamma}(a_m) }{1-\gamma\delta_m}.$$
From the expression in \eqref{wq}, $\delta_m \rightarrow 0$ as $m\rightarrow \infty$ and therefore $1 -\gamma \delta_m > 0$ which is one of the requirement in  Lemma \ref{3}. Clearly the lemma will be proved if we show
$$\lim_{m \rightarrow \infty } K_m = 0.$$
First of all notice that $K_m > 0$. Now using Lemma \ref{point} and $f_m^{2\gamma}(t) \leq t$, we get 
$$K_m \leq -2 \log(a_m) + \frac{\delta_m a_m \gamma}{1-\gamma\delta_m}.$$
Using \eqref{wq}, it implies 
$$K_m \leq  \frac{1}{\gamma_m}\left[  -2\log(a_m)\gamma_m + \gamma a_m \left\{1 - \left( 1 -\frac{2\log(a_m)}{a_m}\right)^{\frac{1}{\gamma}}  \right\}  \right], $$
where $\gamma_m = 1 -\gamma \delta_m = 1-\gamma \left\{1 - \left( 1 -\frac{2\log(a_m)}{a_m}\right)^{\frac{1}{\gamma}}  \right\}.$ Note that $\gamma_m \rightarrow 1$ as $m \rightarrow \infty.$
The lemma will be proved if we show that the function, as $x \rightarrow \infty,$
$$ -2\log(x) + \gamma x \left\{1 - \left( 1 -\frac{2\log(x)}{x}\right)^{\frac{1}{\gamma}}  \right\}  
+ 2 \gamma  \log(x) \left\{ 1-\left( 1 -\frac{2\log(x)}{x}\right)^{\frac{1}{\gamma}} \right\}\rightarrow 0. $$
First let us consider the third  part
\begin{equation*}
\lim_{x \rightarrow \infty} \log(x) \left\{ 1-\left( 1 -\frac{2\log(x)}{x}\right)^{\frac{1}{\gamma}} \right\} = 
\lim_{x \rightarrow \infty} \frac{1-\left( 1 -\frac{2\log(x)}{x}\right)^{\frac{1}{\gamma}}}{(\log(x))^{-1}}.
\end{equation*}
Using L'Hospital's rule, we obtain 
\begin{equation*}
\lim_{x \rightarrow \infty}  \frac{1-\left( 1 -\frac{2\log(x)}{x}\right)^{\frac{1}{\gamma}}}{(\log(x))^{-1}}= \frac{2}{\gamma}
\lim_{x \rightarrow \infty}\left( 1- \frac{2\log(x)}{x}\right)^{\frac{1}{\gamma} -1} \frac{\log^2(x)(\log(x)-1)}{x} \rightarrow 0.
\end{equation*}
We consider the first and the second term together now,
\begin{eqnarray*}
\lim_{x\rightarrow \infty} -2\log(x) + \gamma x \left\{1 - \left( 1 -\frac{2\log(x)}{x}\right)^{\frac{1}{\gamma}}  \right\}  = \lim_{x \rightarrow \infty}
 \frac{\left[ \frac{-2\log(x)}{x} + \gamma \left\{1 - \left( 1 -\frac{2\log(x)}{x}\right)^{\frac{1}{\gamma}}  \right\}\right]}{x^{-1}}.
\end{eqnarray*}
Using L'Hospital rule again 
\begin{eqnarray*}
\lim_{x\rightarrow \infty} \frac{\left[ \frac{-2\log(x)}{x} + \gamma \left\{1 - \left( 1 -\frac{2\log(x)}{x}\right)^{\frac{1}{\gamma}}  \right\}\right]}{x^{-1}} \nonumber\\
= 2 \lim_{x \rightarrow \infty } (\log(x) -1)\left\{ 1 - \left(  1 -2 \frac{\log(x)}{x}\right)^{\frac{1}{\gamma}-1} \right\} = 0.
\end{eqnarray*}
To see the last equality one has to to use the following inequality, to show that the limiting value is less than or equal to $0$: If $\mu > 0$, then for small $x> 0$
$$(1-x)^{\mu} \geq 1-(\mu +1)x.$$
The other side of the follows since for large $x$, it implies that 
$$ \left(\log(x) -1\right)\left\{ 1 - \left(  1 -2 \frac{\log(x)}{x}\right)^{\frac{1}{\gamma}-1} \right\} > 0.$$
Summing up the  inequalities in Step 1 and Step 2, we get 
$$\lim_{m \rightarrow \infty } J_\beta(\tilde{f_m}) = \lim_{ m \rightarrow \infty} I_\beta(f_m) \leq 1+e.  $$
Since $f_m$ is any arbitrary sequence this finishes the proof of the lemma.
\end{proof}
\smallskip

Now we present the proof of the main theorem. \bigskip

\noi\underline{ Proof of Theorem \ref{pro} :  }\smallskip

\noi  If possible let  $J_\beta(\tilde{g}_m)$ does not converges to $J_\beta(\tilde{g}_0)$, where $\tilde{g}_m, \ \tilde{g}_0$ is as in \eqref{xc}. Then from previous lemma's we know that 
$$M_\beta= \lim_{m \rightarrow \infty } J_\beta(\tilde{g}_m) \leq 1 + e.$$
If we show that, there exist some $\phi \in \tilde{\Lambda_1}$ such that $I_\beta(\phi) > 1 +e$, then  clearly $M_\beta > 1+e$  and this would be a contradiction. Consider the function $f \in \Lambda_1$, defined as

\begin{equation*}\label{Main Problemg}
  f(t) = \left\{
 \begin{array}{ll}
 \frac{t}{2} \hspace{8.7mm}&\textrm{in} \ 0 \leq t \leq 2 , \\
  (t-1)^{\frac{1}{2}} \hspace{16.9mm} &\textrm{on}  \  2 \leq t \leq e^2+1,\\ 
  e  \hspace{17.9mm} &\textrm{on}  \     t \geq e^2+1.
\end{array}
\right.
\end{equation*}
Set  $\phi = f^{1-\beta}$. It has been verified in \cite{Carleson-Chang} that $f \in \tilde{\Lambda_1}$ and
$$I_\beta (f) = \int_{0}^\infty e^{f^2(t)-t}dt = 1 +e + \varsigma^*> 1 +e $$
for some $\varsigma^* > 0.$ 
We are left to verify that $\phi  \in  \tilde{\Lambda_1}$. Since $\phi^{'} = 0$ for $t \geq e^2+1$, we have
$$ \int_0^\infty \frac{|\phi^{'}|^2t^\beta}{1-\beta} dt =   \int_0^2\frac{|\phi^{'}|^2t^\beta}{1-\beta} dt +  \int_2^{e^2+1} \frac{|\phi^{'}|^2t^\beta}{1-\beta} dt := I_1 + I_2.$$
Now after simple calculation one obtains 
\begin{eqnarray*}\label{ss}
I_1 =  (1- \beta) \int_0^2 f^{-2\beta}|f'|^2t^\beta dt = 2^{\beta -1}
\end{eqnarray*}
and 
\begin{eqnarray*}
I_2 = \frac{(1- \beta)}{4} \int_2^{e^2+1} (t-1)^{-\beta -1}t^\beta dt= \frac{(1- \beta) }{4} \int_1^{e^2} \frac{(m+1)^\beta}{m^{\beta+1}}dm 
\leq \frac{(1- \beta) }{2}.
\end{eqnarray*}
In the last step above we have used that  for $\beta >0$, it implies that
$ \psi(\beta) \leq \psi(0) = 2$ where $$\psi(\beta) = \int_1^{e^2} \frac{(m+1)^\beta}{m^{\beta+1}}dm.$$
Therefore 
$$\int_0^\infty \frac{|\phi^{'}|^2t^\beta}{1-\beta} dt  =   2^{\beta -1} +  \frac{1-\beta}{2} \leq 1, \  \forall \beta \in [0,1).$$
This proves that $f \in  \tilde{\Lambda_1}$.

\bigskip

%As mentioned in the introduction the optimal logarithmic Moser-Trudinger inequality for the case $\beta =1$ is  also obtained by  Calanchi-Ruf in \cite{ruc}. It would be interesting  to know if  the supremum is attained for this inequality.
\bigskip

\noindent\textbf{Acknowledgments} The research work of the author is supported by ``Innovation in Science Pursuit for Inspired Research (INSPIRE)" under the IVR Number: 20140000099.


\begin{thebibliography}{200}
\small{



\bibitem{Adi-Sandeep} Adimurthi, K. Sandeep,  A singular Moser-Trudinger embedding and its applications, \textit{NoDEA Nonlinear Differential Equations Appl}, 13 (2007), 585--603.
\vspace{-.25cm}



\bibitem{Adi-Tintarev} Adimurthi, C. Tintarev, On compactness in the Trudinger-Moser inequality. \textit{Ann. Sc. Norm. Super. Pisa Cl. Sci}, 13 (2014), 399--416.
\vspace{-.25cm}

%\bibitem{Axler-Bourdon-Ramey}Axler S., Bourdon P. and Ramey W., \textit{ Harmonic function theory}, Second edition, Graduate Texts in Mathematics, 137, Springer-Verlag, New York, 2001.
%\vspace{-.25cm}

\bibitem{ca}   M. Calanchi,  B. Ruf, Some weighted inequalities of Trudinger–Moser type,  \textit{ Progress in Nonlinear Differential Equations and Appl}, vol. 85, Birkh\"auser, 2014, pp. 163--174.
\vspace{-.25cm}

\bibitem{cr}  M. Calanchi,  B. Ruf, Trudinger-Moser type inequalities with logarithmic weights in dimension $N$.\textit{
Nonlinear Anal,} 121 (2015), 403--411. 
\vspace{-.25cm}
\bibitem{ruc}   M. Calanchi,  B. Ruf, On Trudinger-Moser type inequalities with logarithmic weights, \textit{Journal of differential equations}, 258 (2015), 1967--1989.
\vspace{-.25cm}


\bibitem{Carleson-Chang} L. Carleson,  S.-Y. A. Chang, On the existence of an extremal function for an inequality by J. Moser, \textit{Bull. Sci. Math}, 110 (1986),  113--127.
\vspace{-.25cm}
\bibitem{Csato-Roy}G. Csat\'o,  P. Roy, Extremal functions for the singular Moser-Trudinger inequality in 2 dimensions, \textit{Calc. Var. Partial Differential Equations},  54 (2015), 2341--2366. 
\vspace{-.25cm}
\bibitem{Csato-Roy2} G. Csat\'o,  P. Roy, The singular Moser-Trudinger inequality on simply connected domain, \textit{Comm. Partial Differential Equations,} \ DOI:10.1080/03605302.2015.1123276.
\vspace{-.25cm}

\bibitem{Flucher}M. Flucher,  Extremal functions for the Trudinger-Moser inequality in 2 dimensions, \textit{Comment. Math. Helvetici}, 67 (1992), 471--497.
\vspace{-.25cm}

\bibitem{lu}   N. Lam,  G. Lu,  A new approach to sharp Moser-Trudinger and Adams type inequalities: a rearrangement-free argument. \textit{J. Differential Equations,} 255 (2013), 298--325.
\vspace{-.25cm}

\bibitem{LIN} K.-C. Lin, Extremal functions for Moser's inequality. \textit{ Trans. Amer. Math. Soc}, 348 (1996), 2663--2671.\vspace{-.25cm}

\bibitem{Lions}P.-L. Lions,
The concentration-compactness principle in the calculus of variations. The limit case. I,
\textit{Rev. Mat. Iberoamericana 1}, (1985), 145--201. 
\vspace{-.25cm}

\bibitem{lu2}  G. Lu,  H. Tang,  Best constants for Moser-Trudinger inequalities on high dimensional hyperbolic spaces. \textit {Adv. Nonlinear Stud,} 13 (2013), 1035--1052.
\vspace{-.25cm}

\bibitem{Malchiodi-Martinazzi}A. Malchiodi,   L. Martinazzi,
Critical points of the Moser-Trudinger functional on a disk, \textit{
J. Eur. Math. Soc. (JEMS)}, \textbf{16} (2014), 893--908.
\vspace{-.25cm}

\bibitem{manchini}  G. Mancini, K. Sandeep, Moser-Trudinger inequality on conformal discs. \textit{Commun. Contemp. Math,} 12(6) (2010), 1055--1068.
\vspace{-.25cm}

 \bibitem{Moseer }J. Moser,  A sharp form of an inequality by N. Trudinger,  \textit{Indiana Univ. Math. J}, 20, (1971),  1077--1092.
 \vspace{-.25cm}
 \bibitem{r} B. Ruf, A sharp Trudinger-Moser type inequality for unbounded domains in $R^2$. \textit{J. Funct. Anal,} 219 (2005), 340--367.
 \vspace{-.25cm}
 
\bibitem{Struwe 1} M. Struwe, Critical points of embeddings of $H_0^{1,n}$ into Orlicz spaces, \textit{Ann. Inst. H. Poincar\'e Anal. Non Lin\'eaire}, 5 (1988),  425--464. 
\vspace{-.25cm}


\bibitem{sandeep3}  G. Mancini,  K. Sandeep,  C. Tintarev,  Trudinger-Moser inequality in the hyperbolic space $H^N$.  \textit {Adv. Nonlinear Anal}, 2 (2013), 309--324.
\vspace{-.25cm}

\bibitem{tru}  N. S. Trudinger,  On embeddings into Orlicz spaces and some applications, \textit{J. Math. Mech}, 17 (1967),
473--484.






%\bibitem{Moser} Moser J., A sharp form of an inequality by N. Trudinger, \textit{Indiana Univ. Math. J.}, \textbf{20}, (1971), no. 11, 1077--1092.%
%\vspace{-.25cm}



%\bibitem{Struwe}Struwe M.,\textit{Variational methods, applications to %nonlinear partial differential equations and Hamiltonian systems}, Fourth %edition, Springer-Verlag, Berlin, 2008.
%\vspace{-.25cm}

%\bibitem{Talenti} Talenti G., Best constant in Sobolev inequality, \textit{Ann. Mat. Pura Appl.} (4) \textbf{110} (1976), 353--372.


%\bibitem{Trudinger}Trudinger N.S., On embeddings into Orlicz spaces and some applications, \textit{J. Math. Mech}, \textbf{17} (1967), 473--484.
}



%\bibitem{Struwe}Struwe M.,\textit{Variational methods, applications to %nonlinear partial differential equations and Hamiltonian systems}, Fourth %edition, Springer-Verlag, Berlin, 2008.
%\vspace{-.25cm}

%\bibitem{Talenti} Talenti G., Best constant in Sobolev inequality, \textit{Ann. Mat. Pura Appl.} (4) \textbf{110} (1976), 353--372.


%\bibitem{Trudinger}Trudinger N.S., On embeddings into Orlicz spaces and some applications, \textit{J. Math. Mech}, \textbf{17} (1967), 473--484.

\end{thebibliography}
\end{document}